\documentclass[letterpaper]{amsart}

\usepackage{amssymb,bm,hyperref}

\renewenvironment{description}
{\begin{list}{}{%
\setlength{\labelwidth}{\leftmargin}%
\advance \labelwidth-\labelsep}}%
{\end{list}}

\newtheorem{thm}{Theorem}
\newtheorem{lem}[thm]{Lemma}
\theoremstyle{remark}
\newtheorem*{ack}{Acknowledgments}

\begin{document}

\title[Exceptional imprimitive $Q$-polynomial schemes with six classes]{Nonexistence of exceptional imprimitive $Q$-polynomial association schemes with six classes}
\author{Hajime Tanaka}
\address{Department of Mathematics, University of Wisconsin, 480 Lincoln Drive, Madison, WI 53706, U.S.A.}
\curraddr{Graduate School of Information Sciences, Tohoku University, 6-3-09 Aramaki-Aza-Aoba, Aoba-ku, Sendai 980-8579, Japan}
\email{htanaka@math.is.tohoku.ac.jp}
\author{Rie Tanaka}
\address{Department of Mathematics, University of Wisconsin, 480 Lincoln Drive, Madison, WI 53706, U.S.A.}
\email{rtanaka@math.wisc.edu}
\begin{abstract}
Suzuki (1998) showed that an imprimitive $Q$-polynomial association scheme with first multiplicity at least three is either $Q$-bipartite, $Q$-antipodal, or with four or six classes.
The exceptional case with four classes has recently been ruled out by Cerzo and Suzuki (2009).
In this paper, we show the nonexistence of the last case with six classes.
Hence Suzuki's theorem now exactly mirrors its well-known counterpart for imprimitive distance-regular graphs.
\end{abstract}

\maketitle

%%%%%%%%%%%%%%%%%%%%%%%%%%%%%%%%%%%%%%%%%%%
%%%%%%%%%%%%%%%%%%%%%%%%%%%%%%%%%%%%%%%%%%%
%%%%%%%%%%%%%%%%%%%%%%%%%%%%%%%%%%%%%%%%%%%
\section{Introduction}

Distance-regular graphs have been extensively studied.
They form the class of \emph{metric} (or $P$-\emph{polynomial}) association schemes, and play a central role in Algebraic Combinatorics.
\emph{Cometric} (or $Q$-\emph{polynomial}) association schemes are the ``dual version'' of distance-regular graphs.
Besides their importance in the theory of association schemes, cometric schemes are of particular interest because of their connections, e.g., to combinatorial/spherical designs, (Euclidean) lattices and also mutually unbiased bases in quantum information theory.
For recent activity on cometric schemes and related topics, see \cite{BB2009EJC,MT2009EJC,DMM2010pre} and the references therein.

Major theorems concerning distance-regular graphs often have their counterparts for cometric schemes, but the proofs can be very different in terms of both techniques and difficulties; compare, e.g., the proofs of the long-standing \emph{Bannai--Ito Conjecture} and its dual \cite{BDKM2009pre,MW2009EJC}.\footnote{In \cite[p.~237]{BI1984B}, Bannai and Ito conjectured that there are only finitely many distance-regular graphs with any given valency $k>2$.}
It is a well-known elementary fact \cite[p.~315]{BI1984B} that an imprimitive distance-regular graph with valency $k>2$ is bipartite or antipodal.
On the cometric side, Suzuki \cite{Suzuki1998JACa} did show in 1998 that an imprimitive cometric scheme with first multiplicity $m>2$ \emph{and} with more than six classes is $Q$-bipartite or $Q$-antipodal, but there remained two cases of open parameter sets with four and six classes, respectively.
The exceptional case with four classes has recently been ruled out by Cerzo and Suzuki \cite{CS2009EJC}.
In this paper, we finally show that the other case with six classes does not occur (Theorem \ref{d=6 does not occur}).
Hence Suzuki's theorem now exactly mirrors the result for imprimitive distance-regular graphs.

The contents of the paper are as follows.
\S\S\ref{sec: association schemes}, \ref{sec: imprimitive cometric schemes} review basic terminology, notation and facts concerning cometric schemes and their imprimitivity.
\S\S\ref{sec: Q}--\ref{sec: proof} are concerned with the exceptional imprimitive cometric scheme with six classes.
\S\ref{sec: Q} deals with the description of its eigenmatrix, and \S\ref{sec: p161+1} with the calculation of a structure constant.
The nonexistence of the scheme is established in \S\ref{sec: proof}.

%%%%%%%%%%%%%%%%%%%%%%%%%%%%%%%%%%%%%%%%%%%
%%%%%%%%%%%%%%%%%%%%%%%%%%%%%%%%%%%%%%%%%%%
%%%%%%%%%%%%%%%%%%%%%%%%%%%%%%%%%%%%%%%%%%%
\section{Association schemes}\label{sec: association schemes}

In this section and the next, we recall necessary definitions and results.
The reader is referred to \cite{BI1984B,BCN1989B,MT2009EJC} for more background material.

Let $X$ be a finite set and $\mathcal{R}=\{R_0,R_1,\dots,R_d\}$ a set of symmetric binary relations on $X$.
For each $i$, let $A_i$ be the adjacency matrix of the graph $(X,R_i)$.
The pair $(X,\mathcal{R})$ is a \emph{symmetric association scheme} \emph{with} $d$ \emph{classes} if
\begin{description}
\item[(AS1)] $A_0=I$, the identity matrix;
\item[(AS2)] $\sum_{i=0}^dA_i=J$, the all ones matrix;
\item[(AS3)] $A_iA_j$ is a linear combination of $A_0,A_1,\dots,A_d$ for $0\leqslant i,j\leqslant d$.
\end{description}
By (AS1) and (AS3), the vector space $\bm{A}$ spanned by the $A_i$ is an algebra; this is the \emph{Bose--Mesner} \emph{algebra} of $(X,\mathcal{R})$.
Note that $\bm{A}$ is semisimple as it is closed under conjugate transposition, so that it has a basis consisting of the primitive idempotents\footnote{Observe that $\frac{1}{|X|}J$ is an idempotent in $\bm{A}$ with rank one, hence must be primitive.} $E_0=\frac{1}{|X|}J,E_1,\dots,E_d$, i.e., $E_iE_j=\delta_{ij}E_i$, $\sum_{i=0}^dE_i=I$.
By (AS2), $\bm{A}$ is closed under entrywise multiplication, denoted $\circ$.
The $A_i$ are the primitive idempotents of $\bm{A}$ with respect to $\circ$, i.e., $A_i\circ A_j=\delta_{ij}A_i$, $\sum_{i=0}^dA_i=J$.

We define $p_{ij}^h$, $q_{ij}^h$ $(0\leqslant i,j,h\leqslant d)$ by the equations
\begin{equation*}
	A_iA_j=\sum_{h=0}^dp_{ij}^hA_h, \quad E_i\circ E_j=\frac{1}{|X|}\sum_{h=0}^dq_{ij}^hE_h.
\end{equation*}
The $p_{ij}^h$ are nonnegative integers.
On the other hand, since each $E_i\circ E_j$ (being a principal submatrix of $E_i\otimes E_j$) is positive semidefinite, it follows that the $q_{ij}^h$ are real and nonnegative.
The \emph{eigenmatrices} $P$, $Q$ are defined by
\begin{equation*}
	A_i=\sum_{j=0}^dP_{ji}E_j, \quad E_i=\frac{1}{|X|}\sum_{j=0}^dQ_{ji}A_j.
\end{equation*}
Note that $P_{0i},P_{1i},\dots,P_{di}$ give the eigenvalues of $A_i$.
Let
\begin{equation*}
	k_i=P_{0i}, \quad m_i=Q_{0i}.
\end{equation*}
It follows that $k_i$ is the valency of $(X,R_i)$ and $m_i=\mathrm{rank}(E_i)$.
The $m_i$ are called the \emph{multiplicities} of $(X,\mathcal{R})$.

We say that $(X,\mathcal{R})$ is \emph{primitive} if the graphs $(X,R_i)$ $(1\leqslant i\leqslant d)$ are connected,
and \emph{imprimitive} otherwise.
Let $I_r$ (resp. $J_r$) denote the $r\times r$ identity (resp. all ones) matrix.
Then
\begin{lem}\label{lem:imprimitivity}
The following are equivalent:
\begin{enumerate}
\item $(X,\mathcal{R})$ is imprimitive.
\item There are $\mathcal{I},\mathcal{J}\subseteq\{0,1,\dots,d\}$ such that $\frac{1}{r}\sum_{i\in\mathcal{I}}A_i=\sum_{i\in\mathcal{J}}E_i=\frac{1}{r}J_r\otimes I_s$ for some integers $r,s\geqslant 2$ and an ordering of $X$.
\end{enumerate}
Moreover, if \textup{(i)}, \textup{(ii)} hold then $r=\sum_{i\in\mathcal{I}}k_i$ and $s=\sum_{i\in\mathcal{J}}m_i$.
\end{lem}
See, e.g., \cite[\S 2.9]{BI1984B}.
Suppose now that $(X,\mathcal{R})$ is imprimitive, so that $\bigcup_{i\in\mathcal{I}}R_i$ is an equivalence relation on $X$.
Then there is a natural structure of a symmetric association scheme $(\tilde{X},\tilde{\mathcal{R}})$ on the set $\tilde{X}$ of all equivalence classes, called a \emph{quotient scheme} of $(X,\mathcal{R})$.
Let $E=\frac{1}{r}J_r\otimes I_s$ be as in (ii) above.
The Bose--Mesner algebra $\tilde{\bm{A}}$ of $(\tilde{X},\tilde{\mathcal{R}})$ is canonically isomorphic to the ``Hecke algebra'' $E\bm{A}E$ (which is also an ideal of $\bm{A}$);
to be more precise, the primitive idempotents $\{\tilde{E}_i\}_{i\in\mathcal{J}}$ of $\tilde{\bm{A}}$ are indexed by $\mathcal{J}$ and satisfy $E_i=\frac{1}{r}J_r\otimes \tilde{E}_i$ $(i\in\mathcal{J})$, and for every $i$ ($0\leqslant i\leqslant d$) we have $EA_i=(\sum_{j\in\mathcal{I}}p_{ij}^i)\cdot\frac{1}{r}J_r\otimes\tilde{A}$ for some adjacency matrix $\tilde{A}$ of $(\tilde{X},\tilde{\mathcal{R}})$.

%%%%%%%%%%%%%%%%%%%%%%%%%%%%%%%%%%%%%%%%%%%
%%%%%%%%%%%%%%%%%%%%%%%%%%%%%%%%%%%%%%%%%%%
%%%%%%%%%%%%%%%%%%%%%%%%%%%%%%%%%%%%%%%%%%%
\section{Imprimitive cometric schemes}\label{sec: imprimitive cometric schemes}

We say that $(X,\mathcal{R})$ is \emph{cometric} (or $Q$-\emph{polynomial}) with respect to the ordering $\{E_i\}_{i=0}^d$ if for each $i$ $(0\leqslant i\leqslant d)$ there is a polynomial $v_i^*$ with degree $i$ such that $Q_{ji}=v_i^*(Q_{j1})$ $(0\leqslant j\leqslant d)$.
Such an ordering is called a $Q$-\emph{polynomial ordering}.
Note that $(X,\mathcal{R})$ is cometric with respect to the above ordering if and only if for all $i,j,h$ $(0\leqslant i,j,h\leqslant d)$ we have $q_{ij}^h=0$ if $i+j<h$ and $q_{ij}^h\ne 0$ if $i+j=h$.
Suppose now that $(X,\mathcal{R})$ is cometric and set
\begin{equation*}
	m=m_1,\quad a_i^*=q_{1i}^i, \quad b_i^*=q_{1,i+1}^i, \quad c_i^*=q_{1,i-1}^i.
\end{equation*}
It follows that
\begin{equation*}
	a_i^*+b_i^*+c_i^*=m,\quad b_0^*=m,\quad a_0^*=b_d^*=c_0^*=0,\quad c_1^*=1.
\end{equation*}
In fact, the $q_{ij}^h$ are entirely determined by the \emph{Krein array}
\begin{equation*}
	\iota^*(X,\mathcal{R})=\{b_0^*,b_1^*,\dots,b_{d-1}^*;c_1^*,c_2^*,\dots,c_d^*\}.
\end{equation*}
The $v_i^*$ are also determined recursively:
\begin{equation}\label{recurrence}
	v_0^*(x)=1,\quad v_1^*(x)=x,\quad xv_i^*(x)=c_{i+1}^*v_{i+1}^*(x)+a_i^*v_i^*(x)+b_{i-1}^*v_{i-1}^*(x).
\end{equation}
Moreover, if we formally define $v_{d+1}^*$ by \eqref{recurrence} with $i=d$ and $c_{d+1}^*=1$, then we find
\begin{equation}\label{v* d+1}
	c_1^*c_2^*\dots c_d^*v_{d+1}^*(x)=\prod_{i=0}^d(x-Q_{i1}).
\end{equation}

It is well known \cite[p.~315]{BI1984B}  that an imprimitive distance-regular graph with valency $k>2$ is bipartite or antipodal.
On the cometric side, Suzuki proved

\begin{thm}[\cite{Suzuki1998JACa}]\label{classification}
If $(X,\mathcal{R})$ is an imprimitive cometric scheme with $Q$-polynomial ordering $\{E_i\}_{i=0}^d$ and $m>2$, then at least one of the following holds:
\begin{enumerate}
\item $(X,\mathcal{R})$ is $Q$-bipartite: $a_i^*=0$ for $0\leqslant i\leqslant d$.
\item $(X,\mathcal{R})$ is $Q$-antipodal: $b_i^*=c_{d-i}^*$ for $0\leqslant i\leqslant d$, except possibly $i=\lfloor \frac{d}{2}\rfloor$.
\item $d=4$ and $\iota^*(X,\mathcal{R})=\{m,m-1,1,b_3^*;1,c_2^*,m-b_3^*,1\}$.
\item $d=6$ and $\iota^*(X,\mathcal{R})=\{ m,m-1,1,b_3^*, b_4^*, 1; \ 1, c_2^*, m-b_3^*, 1, c_5^*, m\}$, where $a_2^*= a_4^* + a_5^*$.
\end{enumerate}
\end{thm}

It should be remarked that, with the notation of Lemma \ref{lem:imprimitivity}, the cases (i)--(iv) above correspond to $\mathcal{J}=\{0,2,4,\dots\}$, $\mathcal{J}=\{0,d\}$, $\mathcal{J}=\{0,3\}$ and $\mathcal{J}=\{0,3,6\}$, respectively.
Case (iii) in the theorem has recently been ruled out by Cerzo and Suzuki \cite{CS2009EJC} based on the integrality conditions of the $P_{ij}$.
In this paper, we prove

\begin{thm}\label{d=6 does not occur}
Case \textup{(iv)} in Theorem \textup{\ref{classification}} does not occur.
In particular, an imprimitive cometric scheme with $m>2$ is $Q$-bipartite or $Q$-antipodal.
\end{thm}

%%%%%%%%%%%%%%%%%%%%%%%%%%%%%%%%%%%%%%%%%%%
%%%%%%%%%%%%%%%%%%%%%%%%%%%%%%%%%%%%%%%%%%%
%%%%%%%%%%%%%%%%%%%%%%%%%%%%%%%%%%%%%%%%%%%
\section{Discussions: $Q$}\label{sec: Q}

For the rest of this paper, we shall always assume that we are in case (iv) of Theorem \ref{classification}, so that $(X,\mathcal{R})$ is a cometric scheme with six classes and Krein array
\begin{equation*}
	\iota^*(X,\mathcal{R})=\{m,m-1,1,b_3^*, b_4^*, 1; \ 1, c_2^*, m-b_3^*, 1, c_5^*, m\},\quad a_2^*= a_4^* + a_5^*.
\end{equation*}
Note that
\begin{equation*}
	a_2^*=m-1-c_2^*,\quad a_4^*=m-1-b_4^*,\quad a_5^*=m-1-c_5^*,\quad a_1^*=a_3^*=a_6^*=0.
\end{equation*}
This section is devoted to the description of $Q$.
First, we routinely obtain
\begin{align*}
	c_2^*v_2^*(x)&=x^2-m, \\
	c_2^*c_3^*v_3^*(x)&=x^3-a_2^*x^2-(m+c_2^*(m-1))x+ma_2^*, \\
	v_4^*(x)&=xv_3^*(x)-v_2^*(x), \\
	c_5^*v_5^*(x)&=(x-a_4^*)v_4^*(x)-b_3^*v_3^*(x), \\
	mc_5^*v_6^*(x)&=(x^2-a_2^*x-c_2^*(m-1))v_4^*(x)-b_3^*(x-a_5^*)v_3^*(x).
\end{align*}
We now describe the $Q_{i1}$ $(0\leqslant i\leqslant 6)$.
\begin{lem}
Concerning \eqref{v* d+1}, it follows that
\begin{align*}
	mc_2^*c_3^*c_5^*v_7^*(x)=&(x^3-a_2^*x^2-(m+c_2^*(m-1))x+ma_2^*-a_5^*c_3^*) \\
	& \times (x^2+c_2^*x-m)(x-m)(x+1).
\end{align*}
\end{lem}

\begin{proof}
By definition,
\begin{equation*}
	v_7^*(x)=xv_6^*(x)-v_5^*(x).
\end{equation*}
Eliminating $v_6^*,v_5^*$ we find
\begin{align*}
	mc_5^*v_7^*(x)=&(x^3-a_2^*x^2-(m+c_2^*(m-1))x+ma_4^*)v_4^*(x) \\
	&-b_3^*(x^2-a_5^*x-m)v_3^*(x) \\
	=&(c_2^*c_3^*v_3^*(x)-ma_5^*)v_4^*(x)-b_3^*c_2^*v_2^*(x)v_3^*(x)+a_5^*b_3^*xv_3^*(x);
\intertext{the replacement $xv_3^*(x)=v_4^*(x)+v_2^*(x)$ in the last term implies:}
	=&(c_2^*v_3^*(x)-a_5^*)(c_3^*v_4^*(x)-b_3^*v_2^*(x)) \\
	=&(c_2^*v_3^*(x)-a_5^*)(c_3^*xv_3^*(x)-mv_2^*(x)).
\end{align*}
Now the result follows from a straightforward calculation.
\end{proof}

Let $x_i=Q_{i1}$ ($0\leqslant i\leqslant 6$).\footnote{It is customary to use $\theta_i^*$ instead of $x_i$, but we decided to avoid introducing too many asterisks.}
We may assume $x_0=m$, $x_6=-1$, and $x_1,x_2,x_3$ ($x_1>x_2>x_3$) are the roots of 
\begin{equation}\label{cubic factor}
	x^3-a_2^*x^2-(m+c_2^*(m-1))x+ma_2^*-a_5^*c_3^*=0,
\end{equation}
and $x_4,x_5$ ($x_4>x_5$) are the roots of  
\begin{equation*}
	x^2+c_2^*x-m=0.
\end{equation*}
In particular, $x_1,x_2,x_3$ satisfy
\begin{align}
	x_1+x_2+x_3&=a_2^*, \label{x1+x2+x3} \\
	x_1x_2+x_1x_3+x_2x_3&=-m-c_2^*(m-1). \label{x1x2+x1x3+x2x3}
\end{align}
The value of the left-hand side in \eqref{cubic factor} at $x=m-1$ equals $-mc_2^*-a_5^*c_3^*<0$, from which it follows that
\begin{equation}\label{m-1<x1<m}
	m-1<x_1<m.
\end{equation}

The other five columns of $Q$ are determined by the recursion \eqref{recurrence}:
\begin{equation}\label{Q}
	Q=\begin{bmatrix}
	1 & m & \frac{m(m-1)}{c_2^*} & m_3 & mm_3-\frac{m(m-1)}{c_2^*} & mm_6 & m_6 \\
	1 & x_1 & \frac{x_1^2-m}{c_2^*} & \frac{a_5^*}{c_2^*} & -\frac{x_1^2-a_5^*x_1-m}{c_2^*} & -\frac{b_4^*x_1}{c_2^*} & -\frac{b_4^*}{c_2^*} \\
	1 & x_2 & \frac{x_2^2-m}{c_2^*} & \frac{a_5^*}{c_2^*} & -\frac{x_2^2-a_5^*x_2-m}{c_2^*} & -\frac{b_4^*x_2}{c_2^*} & -\frac{b_4^*}{c_2^*} \\
	1 & x_3 & \frac{x_3^2-m}{c_2^*} & \frac{a_5^*}{c_2^*} & -\frac{x_3^2-a_5^*x_3-m}{c_2^*} & -\frac{b_4^*x_3}{c_2^*} & -\frac{b_4^*}{c_2^*} \\
	1 & x_4 & -x_4 & -\frac{m}{c_3^*} & -\frac{b_3^*x_4}{c_3^*} & \frac{b_3^*x_4}{c_3^*} & \frac{b_3^*}{c_3^*} \\
	1 & x_5 & -x_5 & -\frac{m}{c_3^*} & -\frac{b_3^*x_5}{c_3^*} & \frac{b_3^*x_5}{c_3^*} & \frac{b_3^*}{c_3^*} \\
	1 & -1  & -\frac{m-1}{c_2^*} & m_3  & -m_3+\frac{m-1}{c_2^*} & -m_6 & m_6 
	\end{bmatrix},
\end{equation}
where $m_3=\frac{m(m-1)}{c_2^*c_3^*}$, $m_6=\frac{(m-1)b_3^*b_4^*}{c_2^*c_3^*c_5^*}$.
In particular:
\begin{equation}\label{|X|}
	|X|=\sum_{i=0}^6m_i=(m+1)(1+m_3+m_6).
\end{equation}
It should be remarked that the validity of \eqref{Q} can be checked by the formula $B_1^*Q^{\mathsf{T}}=Q^{\mathsf{T}}\mathrm{diag}(x_0,x_1,\dots,x_6)$, where $B_1^*$ is the tridiagonal matrix defined by $(B_1^*)_{ij}=q_{1i}^j$ (cf. \cite[p.~91]{BI1984B}).

%%%%%%%%%%%%%%%%%%%%%%%%%%%%%%%%%%%%%%%%%%%
%%%%%%%%%%%%%%%%%%%%%%%%%%%%%%%%%%%%%%%%%%%
%%%%%%%%%%%%%%%%%%%%%%%%%%%%%%%%%%%%%%%%%%%
\section{Discussions: $p_{16}^1+1$}\label{sec: p161+1}

The section is devoted to the computation of $p_{16}^1+1$.
It should be remarked that there is a formula which expresses the $p_{ij}^h$ as rational functions of the $Q_{ij}$ (see \cite[p.~65]{BI1984B}).
However, the required calculation turns out to be extremely involved, so that we provide a more conceptual (computer-free) approach looking at the quotient scheme $(\tilde{X},\tilde{\mathcal{R}})$.\footnote{We did, however, double-check the formula \eqref{p161+1} below using a software package MAGMA (\href{http://magma.maths.usyd.edu.au/magma/}{http://magma.maths.usyd.edu.au/magma/}).}

We have $\mathcal{J}=\{0,3,6\}$ with the notation of Lemma \ref{lem:imprimitivity}, so that $(\tilde{X},\tilde{\mathcal{R}})$ has two classes.
Since $E_0+E_3+E_6=\frac{1}{m+1}(A_0+A_6)$ by \eqref{Q} and \eqref{|X|}, we find $\mathcal{I}=\{0,6\}$ and
\begin{equation*}
	r=k_6+1=m+1.
\end{equation*}
Note also that $E_0,E_3,E_6$ are linear combinations of $A_0+A_6,A_1+A_2+A_3,A_4+A_5$.
Hence we may write
\begin{equation*}
	A_1+A_2+A_3=J_r\otimes\tilde{A}, \quad A_4+A_5=J_r\otimes\tilde{A}',
\end{equation*}
where $\tilde{A},\tilde{A}'$ are the nontrivial adjacency matrices of $(\tilde{X},\tilde{\mathcal{R}})$.
Let $k$ be the valency of the graph corresponding to $\tilde{A}$.
Then

\begin{lem}
The following hold:
\begin{enumerate}
\item $k_1+k_2+k_3=k(m+1)$.
\item $k_1x_1+k_2x_2+k_3x_3=0$.
\item $k_1x_1^2+k_2x_2^2+k_3x_3^2=km(m+1)$.
\end{enumerate}
\end{lem}

\begin{proof}
(i) is the (constant) row sum of $A_1+A_2+A_3=J_r\otimes\tilde{A}$.
For (ii), observe
\begin{align*}
	0&=(PQ)_{01}=(k_1x_1+k_2x_2+k_3x_3)+(k_4x_4+k_5x_5), \\
	0&=c_2^*c_3^*(PQ)_{05}=-b_4^*c_3^*(k_1x_1+k_2x_2+k_3x_3)+b_3^*c_2^*(k_4x_4+k_5x_5).
\end{align*}
Since $b_3^*c_2^*+b_4^*c_3^*>0$ we obtain (ii), as well as $k_4x_4+k_5x_5=0$; the latter, together with $(PQ)_{02}=0$ and (i) above, implies (iii).
\end{proof}

Solving (i)--(iii) above for $k_1,k_2,k_3$ we routinely obtain, in particular,
\begin{equation*}
	k_1=\frac{k(m+1)(x_2x_3+m)}{(x_1-x_2)(x_1-x_3)}.
\end{equation*}
Note that $A_1$ has constant row sum $k_1=k(p_{16}^1+1)$, from which it follows that
\begin{equation}\label{p161+1}
	p_{16}^1+1=\frac{(m+1)(x_2x_3+m)}{(x_1-x_2)(x_1-x_3)}.
\end{equation}
Moreover, using \eqref{x1+x2+x3}, \eqref{x1x2+x1x3+x2x3} we obtain
\begin{gather}
	x_2x_3+m=x_1^2-a_2^*x_1-c_2^*(m-1), \\
	0<(x_1-x_2)(x_1-x_3)=3x_1^2-2a_2^*x_1-m-c_2^*(m-1). \label{denominator}
\end{gather}

%%%%%%%%%%%%%%%%%%%%%%%%%%%%%%%%%%%%%%%%%%%
%%%%%%%%%%%%%%%%%%%%%%%%%%%%%%%%%%%%%%%%%%%
%%%%%%%%%%%%%%%%%%%%%%%%%%%%%%%%%%%%%%%%%%%
\section{Proof of Theorem \ref{d=6 does not occur}}\label{sec: proof}

We are now ready to prove Theorem \ref{d=6 does not occur}.

Suppose first $p_{16}^1=0$.
Then \eqref{p161+1}--\eqref{denominator} would imply
\begin{align*}
	0 &= (m+1)(x_1^2-a_2^*x_1-c_2^*(m-1))-(3x_1^2-2a_2^*x_1-m-c_2^*(m-1)) \\
	&=(x_1-m)((m-2)x_1-1+c_2^*(m-1)).
\end{align*}
However, since
\begin{equation*}
	\frac{1-c_2^*(m-1)}{m-2}<\frac{1}{m-2}<m-1,
\end{equation*}
this contradicts \eqref{m-1<x1<m}.
Hence $p_{16}^1\geqslant 1$.

Fix a scalar $\alpha$ satisfying
\begin{equation*}
	\frac{m(m+1)}{m^2+1}<\alpha<\min\left\{\frac{m+1}{3},2\right\}.
\end{equation*}
Since $p_{16}^1+1\geqslant 2>\alpha$, by \eqref{p161+1}--\eqref{denominator} we find
\begin{equation*}
	(m+1)(x_1^2-a_2^*x_1-c_2^*(m-1))\geqslant\alpha (3x_1^2-2a_2^*x_1-m-c_2^*(m-1)),
\end{equation*}
or equivalently,
\begin{equation}\label{1st inequality}
	(m+1-3\alpha)x_1^2-(m+1-2\alpha)a_2^*x_1-(m+1-\alpha)c_2^*(m-1)+m\alpha\geqslant 0.
\end{equation}
On the other hand, since $ma_2^*-a_5^*c_3^*\geqslant 0$ it follows from \eqref{cubic factor} that
\begin{equation*}
	x_1^3-a_2^*x_1^2-(m+c_2^*(m-1))x_1\leqslant 0.
\end{equation*}
Multiplying both sides by $-\frac{m+1-3\alpha}{x_1}<0$ we find
\begin{equation}\label{2nd inequality}
	-(m+1-3\alpha)(x_1^2-a_2^*x_1-m-c_2^*(m-1))\geqslant 0.
\end{equation}
Combining \eqref{1st inequality}, \eqref{2nd inequality} we find
\begin{equation*}
	\alpha a_2^*(m-1)\leqslant\alpha a_2^*x_1\leqslant m(m+1)-2\alpha c_2^*(m-1)-2m\alpha,
\end{equation*}
from which it follows that
\begin{equation*}
	0<\alpha c_2^*(m-1)\leqslant m(m+1)-\alpha(m^2+1).
\end{equation*}
However, since the right-hand side above is negative, this is absurd.
The proof is complete.

\bigskip
%%%%%%%%%%%%%%%%%%%%%%%%%%%%%%%%%%%%%%%%%%%
%%%%%%%%%%%%%%%%%%%%%%%%%%%%%%%%%%%%%%%%%%%
%%%%%%%%%%%%%%%%%%%%%%%%%%%%%%%%%%%%%%%%%%%
\begin{ack}
RT would like to thank Akihiro Munemasa for discussions and encouragement.
HT is supported by the JSPS Excellent Young Researchers Overseas Visit Program. 
\end{ack}


\begin{thebibliography}{99}

\bibitem{BDKM2009pre}
S. Bang, A. Dubickas, J. H. Koolen and V. Moulton,
There are only finitely many distance-regular graphs of fixed valency greater than two,
preprint;
arXiv:\href{http://arxiv.org/abs/0909.5253}{0909.5253}. 

\bibitem{BB2009EJC}
E. Bannai and E. Bannai,
A survey on spherical designs and algebraic combinatorics on spheres,
European J. Combin. 30 (2009) 1392--1425.

\bibitem{BI1984B}
E. Bannai and T. Ito,
Algebraic combinatorics I: Association schemes,
Benjamin/Cummings, Menlo Park, CA, 1984.

\bibitem{BCN1989B}
A. E. Brouwer, A. M. Cohen and A. Neumaier,
Distance-regular graphs,
Springer-Verlag, Berlin, 1989.

\bibitem{CS2009EJC}
D. R. Cerzo and H. Suzuki,
Non-existence of imprimitive $Q$-polynomial schemes of exceptional type with $d=4$,
European J. Combin. 30 (2009) 674--681.

\bibitem{DMM2010pre}
E. R. van Dam, W. J. Martin and M. E. Muzychuk,
Uniformity in association schemes and coherent configurations: cometric $Q$-antipodal schemes and linked systems,
preprint;
arXiv:\href{http://arxiv.org/abs/1001.4928}{1001.4928}.

\bibitem{MT2009EJC}
W. J. Martin and H. Tanaka,
Commutative association schemes,
European J. Combin. 30 (2009) 1497--1525;
arXiv:\href{http://arxiv.org/abs/0811.2475}{0811.2475}.

\bibitem{MW2009EJC}
W. J. Martin and J. S. Williford,
There are finitely many $Q$-polynomial association schemes with given first multiplicity at least three,
European J. Combin. 30 (2009) 698--704.

\bibitem{Suzuki1998JACa}
H. Suzuki,
Imprimitive $Q$-polynomial association schemes,
J. Algebraic Combin. 7 (1998) 165--180.

\end{thebibliography}
\end{document}